\documentclass[11pt,a4paper]{article}
\usepackage[utf8]{inputenc}
\usepackage{geometry}
\usepackage{amsmath}
\usepackage{amsfonts}
\usepackage{amsthm}
\usepackage{amssymb}
\usepackage{esint}

\newcommand{\bu}{\mathbf{u}}
\newcommand{\bh}{\mathbf{H}}
\newcommand{\bv}{\mathbf{v}}

\let\div\relax
\DeclareMathOperator{\div}{div}

\newtheorem{definition}{Definition}
\newtheorem{theorem}[definition]{Theorem}
\newtheorem{lemma}[definition]{Lemma}
\newtheorem{remark}[definition]{Remark}

\begin{document}
\title{Liouville type theorem for the stationary equations of magneto-hydrodynamics }
\author{Simon Schulz \footnote{\texttt{simon.schulz1@maths.ox.ac.uk}, \newline Mathematical Institute, University of Oxford, \newline Woodstock Road, OX2 6GG, Oxford, United Kingdom.}}
\maketitle
\abstract{We show that any smooth solution $(\bu,\bh)$ to the stationary equations of magneto-hydrodynamics belonging to both spaces $L^6(\mathbb{R}^3)$ and $BMO^{-1}(\mathbb{R}^3)$ must be identically zero. This is an extension of previous results, all of which systematically required stronger integrability and the additional assumption $\nabla\bu , \nabla \bh \in L^2(\mathbb{R}^3)$, i.e., finite Dirichlet integral.

\medskip

\noindent \texttt{Theme:} Partial differential equations. \newline
\texttt{AMS classification codes:} 35B53, 35Q30, 76W05. \newline
\texttt{Keywords:} Liouville theorem; Caccioppoli inequality; Navier-Stokes equations; MHD.}

\section{Introduction}
Liouville type theorems arise naturally when considering the regularity of solutions to the incompressible Navier-Stokes equations. Development in this direction has been led most notably by Chae, Nadirashvili, Seregin, and \v{S}ver\'{a}k (\textit{c.f.}~\cite{ns chae,seregin sverak,seregin}). Intimately tied to the Navier-Stokes equations are the equations of magneto-hydrodynamics (MHD). The latter system models the motion of an incompressible fluid whose velocity field is affected by magnetic interactions, e.g., the movement of a magnetized plasma.

\medskip

Liouville type theorems have been known to hold for the MHD system, as demonstrated by the works \cite{chae,zyq}. In \cite{chae}, Chae proved that if a smooth solution of the stationary MHD equations is bounded in $L^3(\mathbb{R}^3)$ and has finite Dirichlet integral, then it is identically zero. Later, in \cite{zyq}, Zhang-Yang-Qiu proved that if a smooth solution of the stationary MHD equations is bounded in $L^{\frac{9}{2}}(\mathbb{R}^3)$ and has finite Dirichlet integral, then it is also identically zero. So far, no result exists without the finite Dirichlet integral assumption $\nabla\bu , \nabla\bh \in L^2(\mathbb{R}^3)$.

\medskip

The focus of this paper is to obtain a Liouville theorem for the equations of stationary MHD without the need for finite Dirichlet integral, and with only an $L^6(\mathbb{R}^3)$ integrability criterion. To this end, we closely follow the scheme outlined by Seregin in \cite{seregin}. In using this approach, we also reprove the original results in \cite{chae} and \cite{zyq} without the requirement $\nabla\bu,\nabla\bh\in L^2(\mathbb{R}^3)$. Although many of the estimates in this work are identical to those in \cite{seregin}, we go through them in detail for the sake of making this paper self-contained.

\newpage

\section{Preliminaries}
In what follows we employ the method of Seregin in \cite{seregin}, which first and foremost involves proving a Caccioppoli type inequality. Although Seregin's paper is concerned with the stationary incompressible Navier-Stokes equations, his proof makes a similar Caccioppoli type inequality hold for the equations of magneto-hydrodynamics. In light of this, we structure our paper in the same way as was done in \cite{seregin}.

\medskip

Below are the equations of stationary MHD. As per usual, $\bu$ is the velocity of the fluid and $\bh$ is the magnetic field.

\begin{equation}\label{eq:mhd}
\left\lbrace\begin{aligned}
&\div \bu = 0, \\
&\bu \cdot \nabla \bu - \Delta \bu +\nabla p= \bh \cdot \nabla \bh, \\
&\div \bh = 0, \\
&\bu\cdot\nabla\bh - \bh \cdot \nabla \bu = \Delta \bh.
\end{aligned}\right.
\end{equation}

\begin{definition}
We say that $\mathbf{f} \in BMO^{-1}(\mathbb{R}^3)$ if there exists a skew-symmetric tensor $d \in BMO(\mathbb{R}^3)$ such that $\mathbf{f} = \div d \iff f_i =  d_{ij,j}$ for $i=1,2,3$.
\end{definition}

\begin{remark}
Observe that the requirement that a vector be the divergence of a skew-symmetric tensor is ``equivalent'' to this vector being equal to a curl. Formally, we have
\begin{equation*}
\begin{aligned}
\mathbf{f} = \div d \text{ for } d=(d)_{i,j=1}^3 \text{ skew-symmetric } & \iff  \mathbf{f} = \nabla \times \mathbf{g} \text{ for } \mathbf{g}= \left(
\begin{array}{c}
d_{23}\\
-d_{13}\\
d_{12}
\end{array}
\right), \\
&\iff \div \mathbf{f} = 0.
\end{aligned}
\end{equation*}
\end{remark}

\begin{remark}
If $d\in BMO(\mathbb{R}^3)$, then
\[ \Gamma(s) := \sup_{x_0\in\mathbb{R}^3, r>0} \bigg( \fint_{B(x_0,r)} |d - [d]_{x_0,r}|^s \; dx \bigg)^{\frac{1}{s}} < \infty, \]
for each $1\leq s < \infty$. Here, $[d]_{x_0,r}$ denotes the mean value of $d$ in the ball $B(x_0,r)$. We will recurrently use the finiteness of this quantity in our later estimates.
\end{remark}
We will begin by showing the following theorem.
\begin{theorem}\label{thm:almost}
Let $(\bu,\bh)$ be a smooth solution of system \eqref{eq:mhd} with $\bu,\bh \in BMO^{-1}(\mathbb{R}^3)$. If we additionally require that $\bu,\bh \in L^{q}(\mathbb{R}^3)$ for $q \in (2,6)$, then $\bu \equiv 0$ and $\bh \equiv 0$.
\end{theorem}
Note that the above covers the cases explored by Chae in \cite{chae} and Zhang-Yang-Qiu in \cite{zyq}. However, unlike them, we do not additionally require $\nabla\bu,\nabla\bh \in L^2(\mathbb{R}^3)$. A supplementary argument will then yield the result claimed in the abstract, which is contained in the theorem underneath.
\begin{theorem}\label{thm:final}
Let $(\bu,\bh)$ be a smooth solution of system \eqref{eq:mhd} with $\bu,\bh \in BMO^{-1}(\mathbb{R}^3)$. If we additionally require that $\bu,\bh \in L^{6}(\mathbb{R}^3)$, then $\bu \equiv 0$ and $\bh \equiv 0$.
\end{theorem}

\section{Proof of the main results}
\subsection{Caccioppoli type inequality}
Much like in \cite{seregin}, we have at the heart of our proof a Caccioppoli type inequality, which we develop in this portion of the paper. We state this inequality below.
\begin{lemma}
Let $(\bu,\bh)$ be a smooth solution to system \eqref{eq:mhd} with $\bu,\bh\in BMO^{-1}(\mathbb{R}^3)$, and let $\bv := \bu+\bh$. Then the Caccioppoli type inequality
\begin{equation}\label{eq:caccioppoli}
\int_{B(x_0,R/2)} \left| \nabla \bv \right|^2 \; dx \leq c R^{1-6/s} \left( \int_{B(x_0,R)} \left| \bv-\bv_0 \right|^s \; dx \right)^{\frac{2}{s}},
\end{equation}
holds for any ball $B(x_0,R)\subset \mathbb{R}^3$, any constant $\bv_0 \in \mathbb{R}^3$, and any $s>2$.
\end{lemma}

\begin{proof}
Begin by adding the two evolution equations together to obtain
\begin{equation}\label{eq:add}
\left\lbrace\begin{aligned}
&\div \bv = 0, \\
&(\bu - \bh) \cdot \nabla \bv - \Delta \bv = -\nabla p. 
\end{aligned}\right.
\end{equation}
Note that, since both $\bu$ and $\bh$ are in $BMO^{-1}(\mathbb{R}^3)$, we know that their difference $\bu-\bh$ and $\bv$ are also $BMO^{-1}$ vector fields. In particular, we know that there exists a skew-symmetric tensor $d \in BMO(\mathbb{R}^3)$ such that $\bu-\bh = \div d$.

\medskip

Take an arbitrary ball $B(x_0,R)$ and a non-negative cut-off function $\varphi \in C^{\infty}_c (B(x_0,R))$ with the properties: $\varphi (x) = 1$ in $B(x_0,\rho)$, $\varphi(x) = 0$ outside of $B(x_0,r)$, and $| \nabla \varphi (x) | \leq c/(r-\rho)$ for any $R/2 \leq \rho < r \leq R$. We let $\bar{d} = d - [d]_{x_0,R}$, where $[d]_{x_0,R}$ is the mean value of $d$ on the ball $B(x_0,R)$. From here on, we write $\bar{\bv} = \bv-\bv_0$, where $\bv_0$ is any constant in $\mathbb{R}^3$.

\medskip

Now, consider the following Dirichlet problem
\begin{equation*}
\left\lbrace
\begin{aligned}
&\div \mathbf{w} = \div (\varphi \bar{\bv}) \quad &&\text{in } B(x_0,r), \\
&\mathbf{w} = 0 \quad &&\text{on } \partial B(x_0,r).
\end{aligned} \right.
\end{equation*}
Since the right-hand side of the equation integrates to zero (by the divergence theorem) and is locally integrable, we deduce from Theorem 3.6 in Chapter 1 of \cite{sereginbook} (or from \cite{bog}) that there exists $\mathbf{w} \in W^{1,s}_0(B(x_0,r))$ solving the above, and for which the following inequality holds for $1 < s < \infty$,
\begin{equation*}
\begin{aligned}
\int_{B(x_0,r)} | \nabla \mathbf{w} |^s \; dx &\leq  c \int_{B(x_0,r)} | \div (\varphi \bar{\bv}) |^s \; dx, \\
&= c \int_{B(x_0,r)}  | \nabla \varphi \cdot \bar{\bv} |^s \; dx, \\
& \leq \frac{c}{(r-\rho)^s} \int_{B(x_0,r)} |\bar{\bv}|^s.
\end{aligned}
\end{equation*}
Here, $c=c(s)$ and is independent of $x_0$ and $R$.

\medskip

Next, we follow the bounds as in \cite{seregin}, i.e., we test the second equation in \eqref{eq:add} against $\varphi \bar{\bv} - \mathbf{w}$, to get
\begin{equation*}
\begin{aligned}
\int_{B(x_0,r)} \varphi |\nabla \bv |^2 \; dx = & -\int_{B(x_0,r)} \nabla \bv : (\nabla \varphi \otimes \bar{\bv}) \; dx + \int_{B(x_0,r)} \nabla \mathbf{w} : \nabla \bv \; dx \\
& - \int_{B(x_0,r)} (\div \bar{d} \cdot \nabla \bv) \cdot \varphi \bar{\bv} \; dx + \int_{B(x_0,r)} (\div \bar{d} \cdot \nabla \bv) \cdot \mathbf{w} \; dx.
\end{aligned}
\end{equation*}
We denote the previous integrals by $I_1, \dots, I_4$.
\begin{remark}
The term involving $\nabla p$ has vanished, since $\varphi \bar{\bv} - \mathbf{w}$ is divergence-free and
\begin{equation*}
\int_{B(x_0,r)} \big(\varphi\bar{\bv}-\mathbf{w}\big)\cdot \nabla p \; dx = -\int_{B(x_0,r)} \div\big(\varphi\bar{\bv}-\mathbf{w}\big) p \; dx =0,
\end{equation*}
where the boundary term has vanished due to the compact support of our test function.
\end{remark}

Now we bound the numbered integrals $I_1, \dots, I_4$.
\begin{equation*}
\begin{aligned}
|I_1| & \leq \int_{B(x_0,r)} |\nabla \bv | |\nabla \varphi | | \bar{\bv} | \; dx, \\
& \leq \frac{c}{r-\rho}\bigg( \int_{B(x_0,r)} | \nabla \bv |^2 \; dx \bigg)^{\frac{1}{2}} \bigg( \int_{B(x_0,r)} | \bar{\bv} |^2 \; dx \bigg)^{\frac{1}{2}}, \\
& \leq \frac{c}{r-\rho}\bigg( \int_{B(x_0,r)} | \nabla \bv |^2 \; dx \bigg)^{\frac{1}{2}} \bigg( \int_{B(x_0,r)} | \bar{\bv} |^{s} \; dx \bigg)^{\frac{1}{s}} \bigg( \int_{B(x_0,r)} dx \bigg)^{\frac{s-2}{2s}}, \\
& = \frac{c}{r-\rho}\bigg( \int_{B(x_0,r)} | \nabla \bv |^2 \; dx \bigg)^{\frac{1}{2}} \bigg( \int_{B(x_0,r)} | \bar{\bv} |^{s} \; dx \bigg)^{\frac{1}{s}} R^{3 \left({\frac{s-2}{2s}}\right)}.
\end{aligned}
\end{equation*}
So, we have
\begin{equation*}
|I_1| \leq \frac{c R^{3\left(\frac{s-2}{2s}\right)}}{r-\rho}\bigg( \int_{B(x_0,r)} | \nabla \bv |^2 \; dx \bigg)^{\frac{1}{2}} \bigg( \int_{B(x_0,r)} | \bar{\bv} |^{s} \; dx \bigg)^{\frac{1}{s}}.
\end{equation*}
Similarly,
\begin{equation*}
\begin{aligned}
|I_2| & \leq \int_{B(x_0,r)} |\nabla \bv |  |\nabla \mathbf{w} | \; dx, \\
& \leq \bigg( \int_{B(x_0,r)} | \nabla \bv |^2 \; dx \bigg)^{\frac{1}{2}} \bigg( \int_{B(x_0,r)} |\nabla \mathbf{w} |^2 \; dx \bigg)^{\frac{1}{2}}, \\
& \leq \bigg( \int_{B(x_0,r)} | \nabla \bv |^2 \; dx \bigg)^{\frac{1}{2}} \bigg( \int_{B(x_0,r)} | \nabla \mathbf{w} |^s \; dx \bigg)^{\frac{1}{s}} \bigg( \int_{B(x_0,r)} dx \bigg)^{\frac{s-2}{2s}}, \\
& \leq \frac{c R^{3\left(\frac{s-2}{2s}\right)}}{r-\rho} \bigg( \int_{B(x_0,r)} | \nabla \bv |^2 \; dx \bigg)^{\frac{1}{2}} \bigg( \int_{B(x_0,r)} |\bar{\bv}|^{s} \; dx \bigg)^{\frac{1}{s}}.
\end{aligned}
\end{equation*}
Note that we have implicitly assumed that $s>2$.

\medskip

For $I_3$ and $I_4$ we need to use the skew-symmetry of $d$.
\begin{equation*}
\begin{aligned}
|I_3| & = \bigg| \int_{B(x_0,r)} \bar{d}_{jm,m} v_{i,j} \varphi \bar{v}_i \; dx \bigg| = \bigg| \int_{B(x_0,r)} \bar{d}_{jm} v_{i,j} \varphi_m \bar{v}_i \; dx \bigg|, \\
& \leq \frac{c}{r-\rho} \bigg( \int_{B(x_0,r)} | \nabla \bv |^2 \; dx \bigg)^{\frac{1}{2}} \bigg( \int_{B(x_0,r)} |\bar{d}|^2 |\bar{\bv}|^2 \; dx \bigg)^{\frac{1}{2}}, \\
& \leq \frac{c}{r-\rho} \bigg( \int_{B(x_0,r)} | \nabla \bv |^2 \; dx \bigg)^{\frac{1}{2}} \bigg( \int_{B(x_0,r)}  |\bar{\bv}|^{s} \; dx \bigg)^{\frac{1}{s}} \bigg( \int_{B(x_0,r)}  |\bar{d}|^{\frac{2s}{s-2}} \; dx \bigg)^{\frac{s-2}{2s}}, \\
& \leq \frac{c R^{3\left(\frac{s-2}{2s}\right)}}{r-\rho}\bigg( \int_{B(x_0,r)} | \nabla \bv |^2 \; dx \bigg)^{\frac{1}{2}} \bigg( \int_{B(x_0,r)}  |\bar{\bv}|^{s} \; dx \bigg)^{\frac{1}{s}}.
\end{aligned}
\end{equation*}
For the fourth integral
\begin{equation*}
\begin{aligned}
|I_4| & = \bigg| \int_{B(x_0,r)} \bar{d}_{jm,m} v_{i,j} w_i \; dx \bigg| = \bigg| \int_{B(x_0,r)} \bar{d}_{jm} v_{i,j} w_{i,m} \; dx \bigg|, \\
& \leq \bigg( \int_{B(x_0,r)} | \nabla \bv |^2 \; dx \bigg)^{\frac{1}{2}} \bigg( \int_{B(x_0,r)} |\bar{d}|^2 |\nabla \mathbf{w}|^2 \; dx \bigg)^{\frac{1}{2}}, \\
& \leq  \bigg( \int_{B(x_0,r)} | \nabla \bv |^2 \; dx \bigg)^{\frac{1}{2}} \bigg( \int_{B(x_0,r)} |\nabla \mathbf{w} |^s \bigg)^{\frac{1}{s}}  \bigg( \int_{B(x_0,r)} |\bar{d}|^{\frac{2s}{s-2}} \; dx \bigg)^{\frac{s-2}{2s}}, \\
& \leq \frac{c R^{3\left(\frac{s-2}{2s}\right)}}{r-\rho} \bigg( \int_{B(x_0,r)} | \nabla \bv |^2 \; dx \bigg)^{\frac{1}{2}} \bigg( \int_{B(x_0,r)} |\bar{\bv}|^{s} \; dx \bigg)^{\frac{1}{s}}.
\end{aligned}
\end{equation*}
In total, we have
\begin{equation*}
\begin{aligned}
\int_{B(x_0,\rho)} |\nabla \bv |^2 \; dx \leq \frac{c R^{3\left(\frac{s-2}{2s}\right)}}{r-\rho}\bigg( \int_{B(x_0,r)} | \nabla \bv |^2 \; dx \bigg)^{\frac{1}{2}}\bigg( \int_{B(x_0,R)} |\bar{\bv}|^{s} \; dx \bigg)^{\frac{1}{s}}.
\end{aligned}
\end{equation*}
Applying a weighted Cauchy-Schwarz inequality, we obtain
\begin{equation*}
\begin{aligned}
\int_{B(x_0,\rho)} |\nabla \bv |^2 \; dx \leq \frac{1}{4} \int_{B(x_0,r)} | \nabla \bv |^2 \; dx + \frac{c R^{3\left(\frac{s-2}{s}\right)}}{(r-\rho)^2}\bigg( \int_{B(x_0,R)} |\bar{\bv}|^{s} \; dx \bigg)^{\frac{2}{s}},
\end{aligned}
\end{equation*}
Suitable iterations then give the following Caccioppoli type inequality
\begin{equation*}
\int_{B(x_0,R/2)} |\nabla \bv|^2 \; dx \leq c R^{3\left(\frac{s-2}{s}\right)-2}\bigg( \int_{B(x_0,R)}|\bar{\bv}|^s \bigg)^{\frac{2}{s}},
\end{equation*}
as required.
\end{proof}
\begin{remark}
The positive constant $c$ is independent of $x_0$ and $R$, and depends only on $s$.
\end{remark}

\subsection{The proof of Theorem \ref{thm:almost}}
The proof of Theorem \ref{thm:almost} rests entirely on the observation that we can make the exponent $1-6/s$ negative in the Caccioppoli type inequality \eqref{eq:caccioppoli}. In view of this, we present our proof.
\begin{proof}[Proof of Theorem \ref{thm:almost}]
Suppose $\bu,\bh \in L^q (\mathbb{R}^3)$ for $2<q<6$, and let $\varepsilon := 6/q-1$. Observe that $\varepsilon>0$, so by choosing $\bv_0=0$ the Caccioppoli type inequality \eqref{eq:caccioppoli} now reads
\begin{equation*}
\int_{B(x_0,R/2)} |\nabla \bv|^2 \; dx \leq c R^{-\varepsilon}|| \bv ||^2_{L^q(\mathbb{R}^3)}.
\end{equation*}
By taking the limit as $R \to \infty$ we recover $\nabla \bv \equiv 0$. This implies that $\bv$ is constant, but since $\bv \in L^q(\mathbb{R}^3)$ we know that this constant must be zero. Hence, $\bu \equiv -\bh$.

\medskip

Using this relation, we know from the first evolution equation for $\bu$ in \eqref{eq:mhd} that
\begin{equation}\label{eq:side}
\left\lbrace\begin{aligned}
&\div \bu = 0, \\
&\Delta \bu = \nabla p.
\end{aligned}\right.
\end{equation}
As before, we can find a $\mathbf{w} \in W^{1,q}_0(B(x_0,r))$ such that $\div \mathbf{w} = \div (\varphi \bar{\bu})$, where $\bar{\bu} = \bu-\bu_0$ for some arbitrary constant $\bu_0$ in $\mathbb{R}^3$. Here, $\varphi$ is the same cut-off function that we used in the proof of the Caccioppoli type inequality. Testing \eqref{eq:side} against $\varphi \bar{\bu} - \mathbf{w}$ we obtain
\begin{equation*}
\int_{B(x_0,r)} \varphi |\nabla \bu |^2 \; dx = -\int_{B(x_0,r)} \nabla \bu : (\nabla\varphi \otimes \bar{\bu}) \; dx + \int_{B(x_0,r)} \nabla \mathbf{w} : \nabla \bu \; dx.
\end{equation*}
Once again, we obtain
\begin{equation*}
\int_{B(x_0,R/2)} |\nabla \bu|^2 \; dx \leq cR^{1-6/q} \bigg(\int_{B(x_0,R)} |\bar{\bu}|^q \; dx\bigg)^{\frac{2}{q}},
\end{equation*}
so choosing $\bu_0 =0$ we get
\begin{equation*}
\int_{B(x_0,R/2)} |\nabla \bu|^2 \; dx \leq cR^{-\varepsilon} || \bu ||^2_{L^q(\mathbb{R}^3)}.
\end{equation*}
Taking the limit as $R \to \infty$ we recover $\bu \equiv 0$, which concludes the proof of the theorem.
\end{proof}

\subsection{The proof of Theorem \ref{thm:final}}
In the case where $s=6$ we cannot argue as we did previously. Putting $s=6$ and $\bv_0=0$ in \eqref{eq:caccioppoli} yields
\begin{equation*}
\int_{B(x_0,R/2)} |\nabla \bv|^2 \; dx \leq c || \bv ||^{2}_{L^6(\mathbb{R}^3)}.
\end{equation*}
Hence, passing to the limit $R \to \infty$ gives the reverse Sobolev inequality
\begin{equation}\label{eq:reverse sobolev}
|| \nabla \bv ||_{L^2(\mathbb{R}^3)} \leq c || \bv ||_{L^6(\mathbb{R}^3)}.
\end{equation}
This is not particularly useful in itself, and does not readily produce a reverse Sobolev inequality for the individual vector fields $\bu$ and $\bh$. Instead, one can pick $s=3$ and $\bv_0 = [\bv]_{x_0,R}$ with the aim of constructing an inequality between maximal functions. This is precisely how the proof of Theorem \ref{thm:final} runs, which we elaborate on in the next few paragraphs.

\begin{proof}[Proof of Theorem \ref{thm:final}]
Firstly recall the Gagliardo-Nirenberg type inequality
\begin{equation}\label{eq:gag}
||\bar{\bv}||_{L^{3}(B(x_0,R))} \leq c || \nabla \bv ||_{L^{\frac{3}{2}}(B(x_0,R))}.
\end{equation}
Now choose $s=3$ and $\bv_0 = [\bv]_{x_0,R}$ in the Caccioppoli type inequality \eqref{eq:caccioppoli}, and couple this with \eqref{eq:gag} to obtain the reverse H\"{o}lder inequality
\begin{equation}\label{eq:reverse holder}
\fint_{B(x_0,R/2)} |\nabla \bv|^2 \; dx \leq c \bigg(\fint_{B(x_0,R)} |\nabla \bv|^{\frac{3}{2}} \; dx \bigg)^{\frac{4}{3}},
\end{equation}
where $c$ is independent of $x_0$ and $R$, as per usual.

\medskip

Define the function $h:= |\nabla \bv|^{\frac{3}{2}} \in L^{\frac{4}{3}}(\mathbb{R}^3)$ and let
\begin{equation*}
M_h (x_0) = \sup_{R>0} \fint_{B(x_0,R)} h(x) \; dx
\end{equation*}
be its Hardy-Littlewood maximal function. Now the reverse H\"{o}lder inequality \eqref{eq:reverse holder} reads
\begin{equation*}
M_{h^{\frac{4}{3}}}(x_0) \leq c M^{\frac{4}{3}}_h(x_0) \qquad \forall x_0 \in \mathbb{R}^3.
\end{equation*}
From the maximal function inequality in $L^p(\mathbb{R}^3)$ for $p>1$ (\textit{c.f.}~\cite{stein}), we know that there exists a universal constant $c_0>0$ such that
\begin{equation*}
\begin{aligned}
\int_{\mathbb{R}^3} M_h^{\frac{4}{3}} (x) \; dx &\leq c_0 \int_{\mathbb{R}^3} h^{\frac{4}{3}} (x) \; dx, \\
&= c_0 \int_{\mathbb{R}^3} |\nabla \bv|^2 \; dx, \\
&\leq c || \bv ||^2_{L^6(\mathbb{R}^3)},
\end{aligned}
\end{equation*}
where the last inequality is exactly \eqref{eq:reverse sobolev}. Thus we have shown that both $h^{\frac{4}{3}}$ and its maximal function $M_{h^{\frac{4}{3}}}$ are $L^1(\mathbb{R}^3)$ functions, which is only possible if $h \equiv 0$ (\textit{c.f.}~\cite{stein}). This implies that $\bv$ is constant, thus once again we arrive at $\bu \equiv -\bh$.

\medskip

We now show that we must have $\bu \equiv 0$. Making use of the relation $\bu \equiv -\bh$ as we did in the proof of Theorem \ref{thm:almost}, we recover \eqref{eq:side} and the Caccioppoli type inequality
\begin{equation*}
\int_{B(x_0,R/2)} |\nabla \bu|^2 \; dx \leq cR^{1-6/q} \bigg(\int_{B(x_0,R)} |\bar{\bu}|^q \; dx\bigg)^{\frac{2}{q}}.
\end{equation*}
Picking $q=6$ and $\bu_0 = 0$ we recover
\begin{equation*}
|| \nabla \bu||_{L^2(\mathbb{R}^3)} \leq c || \bu ||_{L^6(\mathbb{R}^3)},
\end{equation*}
as expected. Selecting $q=3$ and $\bu_0 = [\bu]_{x_0,R}$ and using the same strategy as before, we arrive at the maximal function inequality
\begin{equation*}
M_{\tilde{h}^{\frac{4}{3}}}(x_0) \leq c M^{\frac{4}{3}}_{\tilde{h}}(x_0) \qquad \forall x_0 \in \mathbb{R}^3,
\end{equation*}
where $\tilde{h} = |\nabla \bu|^{\frac{3}{2}}$. The same argument as before then yields $\bu \equiv 0$, as required.
\end{proof}

\section*{Acknowledgements}
This work was supported by the Engineering and Physical Sciences Research Council grant [EP/L015811/1]. The author wishes to thank Gui-Qiang Chen and Gregory Seregin for useful discussions.


\newpage


\begin{thebibliography}{99}

%Altman M. Contractor directions and monotone operators. J Integ Eq, 1979, \textbf{20}(2):17-33

\bibitem{bog}
Bogovski\u{\i}, M E. Solution of the first boundary value problem for an equation of continuity of an incompressible medium. Dokl Akad Nauk SSSR, 1979, \textbf{248}(5): 1037--1040

\bibitem{ns chae}
Chae D.
Liouville type theorems for the Euler and Navier-Stokes equations.
Advances in Mathematics, 2011, \textbf{228}: 2855--2868

\bibitem{chae}
Chae D, Weng S.
Liouville type theorems for the steady axially symmetric
Navier-Stokes and magnetohydrodynamic equations.
Discrete Contin Dyn Syst, 2016, \textbf{36}(10): 5267--5285

%% Books %% initial(s) surname(s).  title.  place of publication:publisher, year

 
\bibitem{gt}
Gilbarg D, Trudinger N S.
Elliptic partial differential equations of second order.
Berlin: Springer-Verlag, 2001

\bibitem{seregin sverak}
Koch G, Nadirashvili N, Seregin G A, \v{S}ver\'{a}k V.
Liouville theorems for the Navier-Stokes equations and applications.
Acta Math, 2009, \textbf{203}: 83--105

\bibitem{sereginbook}
Seregin G A.
Lecture notes on regularity theory for the Navier-Stokes equations. 
Hackensack, NJ: World Scientific Publishing Co Pte Ltd, 2015 

\bibitem{seregin}
Seregin G A.
Liouville type theorem for stationary Navier-Stokes equations.
Nonlinearity, 2016, \textbf{29}: 2191--2195

\bibitem{stein}
Stein E M.
Singular Integrals and Differentiability Properties of Functions.
Princeton, NJ: Princeton University Press, 1970

\bibitem{zyq}
Zhang Z, Yang X, Qiu S.
Remarks on Liouville Type Result
for the 3D Hall-MHD System.
J Part Diff Eq, 2015, \textbf{28}(3): 286--290
\end{thebibliography}
\end{document}